\documentclass[11pt]{article}

\usepackage[T1]{fontenc}
\usepackage[utf8]{inputenc}
\usepackage[english]{babel}
\usepackage{amsmath,amssymb,amsthm,mathtools}
\usepackage{graphicx}
\usepackage{enumerate}
\usepackage{microtype}
\usepackage{geometry}
\usepackage{hyperref}
\hypersetup{colorlinks=true,linkcolor=blue,citecolor=blue,urlcolor=blue}

\newtheorem{theorem}{Theorem}[section]
\newtheorem{proposition}[theorem]{Proposition}
\newtheorem{lemma}[theorem]{Lemma}
\newtheorem{corollary}[theorem]{Corollary}
\newtheorem{conjecture}[theorem]{Conjecture}
\theoremstyle{remark}
\newtheorem{remark}[theorem]{Remark}

\newtheorem{problem}[theorem]{Problem}
\theoremstyle{definition}
\newtheorem{definition}[theorem]{Definition}

\newcommand{\C}{\mathbb C}
\newcommand{\R}{\mathbb R}
\newcommand{\Q}{\mathbb Q}
\newcommand{\Qbar}{\overline{\mathbb Q}}
\newcommand{\Z}{\mathbb Z}
\newcommand{\F}{\mathbb F}
\newcommand{\PP}{\mathbb P}

\newcommand{\Gal}{\operatorname{Gal}}
\newcommand{\Frob}{\operatorname{Frob}}
\newcommand{\Deck}{\operatorname{Deck}}
\newcommand{\End}{\operatorname{End}}
\newcommand{\ord}{\operatorname{ord}}

\newcommand{\Div}{\operatorname{Div}}
\newcommand{\SL}{\operatorname{SL}}
\newcommand{\PSL}{\operatorname{PSL}}
\newcommand{\GL}{\operatorname{GL}}
\newcommand{\PGL}{\operatorname{PGL}}
\newcommand{\Tr}{\operatorname{Tr}}

\title{Peels of Equilateral Triangulated Surfaces}
\author{Alberto Verjovsky\thanks{Instituto de Matem\'aticas, Unidad Cuernavaca, Universidad Nacional Aut\'onoma de M\'exico (UNAM), Av. Universidad s/n, Col. Lomas de Chamilpa, 62210 Morelos, M\'exico. Email: \texttt{albertoverjovsky@gmail.com}}}
\date{}

\begin{document}
\maketitle

\begin{abstract}
This paper studies the modularity problem for elliptic curves from the elementary geometry of peels of decorated equilateral triangulations. A peel is a finite developing domain together with the boundary identifications that reconstruct the triangulated surface. In genus one it recovers the universal cover and its rank-two deck lattice. For a fixed prime $\ell$, the finite unramified coverings $[\ell^n]:E\to E$ form a compatible tower whose deck groups are $E[\ell^n]$; their inverse limit is the Tate module $T_\ell(E)$, and over $\Q$ the compatible Galois action gives the usual $\ell$-adic representation. Cyclic isogenies, by contrast, give the index-$q$ lattice neighbors that form the local model for Hecke and Brandt correspondences. For a semistable elliptic curve $E/\Q$ of prime conductor $p$, the missing reciprocity is the construction, without using modularity, of a nonzero class $c_E$ in the degree-zero supersingular Brandt module such that $B_qc_E=a_q(E)c_E$ at every good prime. I make this obstruction explicit and study oriented CM packets attached to the Frobenius discriminants $D_q=a_q(E)^2-4q$. New elementary results show that $D_q$ is a nonsquare modulo $p$ exactly when the residual Frobenius polynomial is irreducible, that any residual image containing such an element gives a positive-density set of primes for which $p$ is inert in the quadratic field determined by $D_q$, and that the corresponding discriminants are necessarily unbounded. I also give an oriented nonvanishing criterion and an ordinary Brandt nonvanishing criterion in terms of a Frobenius-conjugate pair, equivalently an irreducible quadratic factor of the Hilbert class polynomial modulo $p$. A precise conditional consequence of $p$-adic CM equidistribution is recorded. At conductor $37$ the construction is completely explicit: the discriminant $-76$ produces the anti-invariant Brandt line and the $2$-Brandt eigenvalue agrees with $a_2(E)$. The final equality for all Hecke operators remains the modularity barrier; no new proof of the Modularity Theorem is claimed.
\end{abstract}



\section{Overture: triangles before arithmetic}

I would like to begin with a picture rather than with arithmetic machinery.  The picture is completely elementary.  Start with an elliptic curve $E$ carrying a finite decorated equilateral triangulation.  Remove the vertices.  What remains is a torus with finitely many punctures, decomposed into triangles whose vertices have disappeared to the ends.  In the Belyi situation these triangles can be read as hyperbolic ideal triangles.

More precisely, let
\[
        f:E\longrightarrow \PP^1
\]
be a Belyi map, and put
\[
        V=f^{-1}\{0,1,\infty\},\qquad E^\circ=E\setminus V.
\]
Then
\[
        f:E^\circ\longrightarrow \PP^1\setminus\{0,1,\infty\}
\]
is an unramified covering.  The thrice-punctured sphere is obtained by gluing two hyperbolic ideal triangles.  Pulling this picture back gives an ideal triangulation of $E^\circ$ by copies of the same ideal triangle.  Each ideal triangle has angles $0,0,0$ and area $\pi$.  It is the limiting case $m=\infty$ of the hyperbolic triangles with angles $\pi/m,\pi/m,\pi/m$ used later in the tessellations $T(m,m,m)$.

Nothing sophisticated is needed to count the pieces.  If $r=|V|$, then
\[
        \chi(E^\circ)=\chi(E)-r=-r.
\]
By Gauss--Bonnet the complete finite-area hyperbolic metric on $E^\circ$ has area
\[
        \operatorname{Area}(E^\circ)=-2\pi\chi(E^\circ)=2\pi r.
\]
Since every ideal triangle has area $\pi$, the triangulation has exactly
\[
        2r
\]
ideal triangles.

For a Belyi map of degree $d$ there is an even simpler count.  If $\sigma_0,\sigma_1,\sigma_\infty$ are the monodromy permutations and $c(\sigma)$ denotes the number of cycles of $\sigma$, then the points of $V$ are counted by
\[
        r=c(\sigma_0)+c(\sigma_1)+c(\sigma_\infty).
\]
Riemann--Hurwitz in genus one gives $r=d$.  Thus
\[
\boxed{\begin{gathered}
\deg f=d\quad\Longrightarrow\quad |V|=d,\\
2d\ \hbox{ideal triangles}\quad\Longrightarrow\quad
\operatorname{Area}(E^\circ)=2\pi d.
\end{gathered}}
\]
This is an elementary geometric consequence of Belyi's theorem in genus one: a curve defined over $\overline{\Q}$ can be represented by a finite gluing of copies of one ideal triangle, with all the complexity transferred to the way the sides are identified.

This elementary picture was one of the reasons for introducing peels.  A peel opens the finite triangulated surface and displays the gluing data.  After the vertices are removed, the same operation may be regarded hyperbolically: one opens a finite union of ideal triangles and records the side identifications.  The finite drawing therefore retains the combinatorics of the dessin, the topology of the punctured surface, and the covering of the thrice-punctured sphere.

\subsection{The main character: a peel}

Since the peel is the object from which the rest of the paper starts, I record its definition here, before using it further.  This is the definition used in the joint work with Jos\'e Juan-Zacar\'ias, expressed so that the Euclidean, hyperbolic, and ideal cases are treated together.

\begin{definition}[Peel]\label{def:peel-overture}
Let $X$ be a compact oriented surface with a decorated triangulation $\mathcal T$.  Fix one of the following model geometries:
\begin{enumerate}[(i)]
\item the Euclidean equilateral triangle;
\item the hyperbolic equilateral triangle $\Delta_{m,m,m}$ with angles $\pi/m$, where $m>2$ is an integer and the hyperbolic case occurs for $m>3$;
\item the limiting case $m=\infty$, namely the ideal triangle $\Delta_{\infty,\infty,\infty}$, whose angles are zero and whose area is $\pi$.
\end{enumerate}
Let $T(m,m,m)$ denote the corresponding triangular tessellation in the hyperbolic cases.  A \emph{peel} of $X$ is a continuous map
\[
        \Psi:P\longrightarrow X,
\]
where $P$ is a connected polygon formed by finitely many triangles of the chosen model tessellation, such that
\begin{enumerate}[(a)]
\item $\Psi$ sends each triangle of $P$ onto a triangle of $\mathcal T$;
\item the restriction of $\Psi$ to each triangle is an isometry;
\item $\Psi$ preserves the decoration;
\item $\Psi$ is injective on the interior of $P$, and every point of $X$ has at most two preimages under $\Psi$.
\end{enumerate}
If $x,y\in\partial P$ satisfy $\Psi(x)=\Psi(y)$, the relation $x\sim y$ records the corresponding boundary identification.  These identifications are part of the peel data.  For a maximal peel all boundary sides are paired and $P/\!\sim$ recovers $X$.
\end{definition}

Thus the phrase used repeatedly in this paper,
\[
        \boxed{\text{a peel is a finite piece of a developing map}},
\]
is an interpretation of Definition~\ref{def:peel-overture}, not a substitute for it.  The definition is finite and elementary: triangles, isometries, decorations, and side identifications.  It is precisely this finite object that will later be asked to carry arithmetic information.

\subsection{Peels of a decorated equilateral triangulation of an elliptic curve and its modulus $\tau$}

Let $X$ be a Riemann surface obtained from a decorated equilateral triangulation.  Its edges and vertices form a decorated graph.  If
\[
        \Psi:P\longrightarrow X
\]
is a peel, write
\[
        \mathcal G_P:=\Psi(\partial P)\subset X
\]
for the boundary graph of the peel.  The complement $X\setminus\mathcal G_P$ is the image of the interior of $P$; for a maximal peel it is a single open $2$-cell.

Recall that two connected graphs have the same homotopy type if and only if their fundamental groups, which are free groups, have the same rank.  For a finite connected graph $G$ this rank is its first Betti number
\[
        \beta_1(G)=1-\chi(G).
\]

\begin{proposition}[The boundary graph of a peel]\label{prop:peel-bouquet}
Let $E$ be an elliptic curve obtained from a decorated equilateral triangulation and let $\Psi:P\to E$ be a maximal peel.  Then $\mathcal G_P$ has the homotopy type of a bouquet of two circles.  In particular,
\[
        \pi_1(\mathcal G_P,x_0)\cong F_2,
\]
for every $x_0\in\mathcal G_P$, where $F_2$ denotes the free group on two generators.
\end{proposition}

\begin{proof}
Let $v$ and $a$ be the numbers of vertices and edges of $\mathcal G_P$.  Since $\mathcal G_P$ is connected,
\[
        \beta_1(\mathcal G_P)=1-\chi(\mathcal G_P)=1-v+a.
\]
The graph $\mathcal G_P$ together with the unique $2$-cell $E\setminus\mathcal G_P$ gives a cell decomposition of the torus.  Therefore
\[
        0=\chi(E)=v-a+1,
\]
and hence
\[
        \beta_1(\mathcal G_P)=1-v+a=2.
\]
A connected graph of first Betti number $2$ is homotopy equivalent to a bouquet of two circles.
\end{proof}

\begin{remark}\label{rem:genus-g-boundary}
The same argument works for a compact oriented surface of genus $g$.  Since
\[
        \chi(X)=2-2g=v-a+1,
\]
one obtains
\[
        \beta_1(\mathcal G_P)=1-v+a=2g.
\]
Thus $\mathcal G_P$ is homotopy equivalent to a bouquet of $2g$ circles.
\end{remark}

The genus-one case contains more information than the rank count.  The boundary graph already carries all of the first homology of the elliptic curve.

\begin{proposition}[The boundary graph carries $H_1$]\label{prop:peel-H1}
Under the hypotheses of Proposition~\ref{prop:peel-bouquet}, the inclusion
\[
        i:\mathcal G_P\hookrightarrow E
\]
induces an isomorphism
\[
        i_*:H_1(\mathcal G_P,\Z)\xrightarrow{\ \cong\ }H_1(E,\Z).
\]
Consequently the induced homomorphism
\[
        i_\#: \pi_1(\mathcal G_P,x_0)\longrightarrow\pi_1(E,x_0)\cong\Z^2
\]
is surjective.  Hence there are loops $\gamma_1,\gamma_2\subset\mathcal G_P$, based at $x_0$, whose classes generate $\pi_1(E,x_0)$.
\end{proposition}

\begin{proof}
The torus $E$ is obtained from its $1$-skeleton $\mathcal G_P$ by attaching one $2$-cell.  Hence the inclusion of the $1$-skeleton induces a surjection
\[
        i_*:H_1(\mathcal G_P,\Z)\twoheadrightarrow H_1(E,\Z).
\]
For example, this follows directly from cellular approximation: every loop in the CW complex $E$ is homotopic to a loop in its $1$-skeleton.  By Proposition~\ref{prop:peel-bouquet},
\[
        H_1(\mathcal G_P,\Z)\cong\Z^2,
\]
while $H_1(E,\Z)\cong\Z^2$ because $E$ is a torus.  A surjective homomorphism between free abelian groups of the same finite rank is an isomorphism.  Therefore $i_*$ is an isomorphism.

Since $\pi_1(E)\cong\Z^2$ is abelian, the homomorphism $i_\#$ factors through the abelianization $H_1(\mathcal G_P,\Z)$.  The isomorphism just proved shows that $i_\#$ is surjective.  We may therefore choose loops $\gamma_1,\gamma_2\subset\mathcal G_P$ whose images generate $\pi_1(E)$.
\end{proof}

We shall also use the standard correspondence between the fundamental group and the deck group.  If
\[
        p:(\widetilde X,\widetilde x_0)\longrightarrow(X,x_0)
\]
is the universal covering, then every class $[\gamma]\in\pi_1(X,x_0)$ determines the unique deck transformation $g_\gamma$ satisfying
\[
        g_\gamma(\widetilde x_0)=\widetilde\gamma(1),
\]
where $\widetilde\gamma$ is the lift of $\gamma$ beginning at $\widetilde x_0$.  This gives the usual isomorphism
\[
        \pi_1(X,x_0)\cong\Deck(\widetilde X/X).
\]

From this point through the discussion of the deck lattice below, $E$ denotes an elliptic curve.  When an ideal triangulation is used, it is understood on the punctured surface $E^\circ=E\setminus V$, while the boundary graph and the homology statements are taken in the compact elliptic curve $E$ obtained by filling the punctures.

The preceding elementary facts recover a modulus directly from a peel.

\begin{theorem}[The peel and a modulus]\label{thm:peel-modulus-overture}
Let $E$ be an elliptic curve obtained from a decorated equilateral triangulation, let
\[
        p:\C\longrightarrow E
\]
be its holomorphic universal covering, and let $\Psi:P\to E$ be a hyperbolic peel.  This includes the ideal model in the sense explained above: the ideal triangulation lives on $E^\circ$, and the boundary graph is considered after compactifying to $E$.  Then there are loops $\gamma_1,\gamma_2\subset\mathcal G_P$ such that the associated deck transformations are translations
\[
        g_{\gamma_1}(z)=z+\omega_1,
        \qquad
        g_{\gamma_2}(z)=z+\omega_2,
\]
and generate $\Deck(p)$.  After ordering the generators so that
\[
        \tau:=\frac{\omega_1}{\omega_2}\in\mathbb H,
\]
one has
\[
        E\cong \C/(\Z\omega_1+\Z\omega_2)
        \cong \C/(\Z+\Z\tau).
\]
Thus the peel boundary graph contains loops from which a modulus $\tau$ of $E$ is recovered.
\end{theorem}

\begin{proof}
By Proposition~\ref{prop:peel-H1}, choose $\gamma_1,\gamma_2\subset\mathcal G_P$ whose images generate $\pi_1(E)\cong\Z^2$.  Through the isomorphism between $\pi_1(E)$ and the deck group of the universal cover, they give two generators of $\Deck(p)$.  A holomorphic automorphism of $\C$ has the form $z\mapsto az+b$ with $a\neq0$; if $a\neq1$ it has a fixed point.  Since a nontrivial deck transformation acts without fixed points, every nontrivial deck transformation is a translation.  Hence the two generators have the displayed form for some linearly independent periods $\omega_1,\omega_2$.  Rescaling by $\omega_2$ gives the lattice $\Z+\Z\tau$, with $\tau\in\mathbb H$ after choosing the orientation of the basis.
\end{proof}

The number $\tau$ obtained in the theorem is not canonical as a point of $\mathbb H$.  Changing the oriented basis of the deck lattice replaces it by $(a\tau+b)/(c\tau+d)$ for a matrix in $\SL_2(\Z)$.  Thus the peel determines a modulus representative, or equivalently the point of $\mathbb H/\SL_2(\Z)$ corresponding to the elliptic curve.

This is the elementary topological content of the peel construction.  Before any Tate module or Hecke operator appears, the peel has already exhibited a free group $F_2$, its abelian quotient $\Z^2$, the universal cover, the deck lattice, and a modulus representative of the elliptic curve.  The arithmetic questions in the second part of the paper ask how much more can be recovered when the triangulation and its covering data are defined over number fields.

There is another elementary feature which will matter later.  An ideal triangulation is not isolated.  If the punctures are kept fixed, one may pass from one ideal triangulation to another by successive diagonal flips: two adjacent triangles form an ideal quadrilateral, and one replaces its diagonal by the other one.  Thus the same punctured elliptic curve carries a whole family of triangulations related by very simple local moves.

There is also a different operation: refinement.  In the Euclidean equilateral picture one may join the midpoints of the sides of each triangle and replace it by four equilateral triangles of half the side length.  Repeating the procedure produces finer and finer triangulations.  On an elliptic curve this midpoint subdivision is closely related to pulling a triangulation back by multiplication maps
\[
        [n]:E\longrightarrow E,
\]
whose degree is $n^2$.  For $n=2$ the fourfold subdivision is already visible in one triangle.  After the new vertices are removed, the puncture set has changed, so refinement is different from a flip; nevertheless both operations belong to the same elementary geometry of triangles on the torus.

These two moves suggest a question which is deliberately stated before introducing any Galois representation, Tate module, Hecke operator, or Brandt module:

\begin{quote}
Given two finite ideal triangulations associated with the same elliptic curve, what information is preserved under flips and refinements, and what information depends on the particular triangulation?
\end{quote}

For Belyi triangulations there is an arithmetic version of the question.  Belyi's theorem says that the existence of such a finite Belyi triangulation characterizes curves defined over $\overline{\Q}$.  In genus one one may therefore ask:

\begin{quote}
What property of the ideal triangulation, of its side-pairings, or of its peel distinguishes those elliptic curves which descend from $\overline{\Q}$ to $\Q$?
\end{quote}

For the underlying elliptic curve this is a question about descent of its complex structure.  By the Belyi pair I mean the pair $(E,f)$; asking that this pair descend is stronger, because the map $f$ itself must also descend.  The point here is not to answer the question at once, but to formulate it in the language of a finite geometric object.

The same picture leads directly to modular curves.  The modular curve $X_0(N)$ has a natural map
\[
        j_N:X_0(N)\longrightarrow X(1)\simeq\PP^1
\]
branched only over the two elliptic values and the cusp.  After normalization of these three values to $0,1,\infty$, this is again a Belyi map.  Put
\[
        X_0(N)^\circ:=X_0(N)\setminus j_N^{-1}\{0,1,\infty\}.
\]
Then $j_N$ restricts to an unramified covering of the thrice-punctured sphere, and $X_0(N)^\circ$ is tiled by ideal triangles coming from the Farey tessellation.  Thus an elliptic Belyi curve and a modular curve are both finite triangulated coverings of the same thrice-punctured sphere:
\[
\begin{array}{ccc}
E^\circ & \longrightarrow & \PP^1\setminus\{0,1,\infty\},\\[2mm]
X_0(N)^\circ & \longrightarrow & \PP^1\setminus\{0,1,\infty\}.
\end{array}
\]
On every ideal triangle the covering map is the standard isometry onto one of the two ideal triangles downstairs.

Consequently one can form a common finite cover by taking the fiber product and normalizing.  Geometrically, one obtains a triangulated surface which maps triangle by triangle to a finite cover of the modular curve and to a finite cover of the elliptic curve.  This construction is elementary and automatic for two Belyi coverings.  It therefore cannot by itself characterize modularity.  That observation is important: the existence of a common triangulated cover is not the theorem we seek.

\subsection{The second main character: the Modularity Theorem}

At this point the second principal object of the paper should be stated explicitly.  The first is the peel.  The second is the Modularity Theorem for elliptic curves over $\Q$.

\begin{theorem}[Modularity Theorem: geometric and arithmetic forms]\label{thm:modularity-overture}
Let $E/\Q$ be an elliptic curve of conductor $N$.  Then there exists a nonconstant morphism defined over $\Q$
\[
        \varphi_E:X_0(N)\longrightarrow E,
\]
where $X_0(N)$ is the compactified modular curve attached to the congruence subgroup
\[
\Gamma_0(N)=\left\{\begin{pmatrix}a&b\\ c&d\end{pmatrix}\in\SL_2(\Z):c\equiv0\pmod N\right\}.
\]  If $\omega_E$ is a nonzero invariant differential on $E$, then for the normalized weight-$2$ newform
\[
        f_E(z)=\sum_{n\geq1}a_n(f_E)q^n,
        \qquad q=e^{2\pi iz},
\]
associated with $E$, one has
\[
        \varphi_E^*\omega_E
        =\kappa_E\,2\pi i\,f_E(z)\,dz
\]
for some nonzero rational constant $\kappa_E$ after a choice of rational invariant differential.  Equivalently, the Hasse--Weil $L$-function of $E$ is the $L$-function of $f_E$:
\[
        \boxed{L(E,s)=L(f_E,s).}
\]
In particular, for every prime $p\nmid N$,
\[
        \boxed{a_p(E)=p+1-\#E(\F_p)=a_p(f_E),}
\]
and the local Euler polynomial is
\[
        1-a_p(E)T+pT^2
        =1-a_p(f_E)T+pT^2.
\]
At the primes dividing $N$, the corresponding local Euler factors agree as part of the equality of the two $L$-functions.
\end{theorem}

This theorem is the point at which the geometric and arithmetic forms of modularity meet.  Geometrically, it gives an actual morphism from the modular curve to the elliptic curve.  Arithmetically, it identifies the Frobenius traces of $E$ with the Hecke eigenvalues of a weight-$2$ newform.  These are two expressions of the same theorem, not two unrelated facts \cite{Wiles,BCDT}.

The geometric form also explains why triangulated surfaces occur naturally in the present problem.  The punctured modular curve and the punctured Belyi elliptic curve are both built from ideal triangles over the thrice-punctured sphere, while modularity supplies the additional map
\[
        X_0(N)\longrightarrow E.
\]
The elementary common triangulated cover described above is automatic; the modular parametrization is not.  The problem is therefore to understand what extra structure, visible in the triangulations, their peels, or their finite refinements, distinguishes the modular situation.

The question which motivated the present paper is sharper:

\begin{quote}
What additional property of this elementary triangulated correspondence distinguishes the case in which the elliptic curve is modular, that is, the case in which there is a nonconstant holomorphic map
\[
        X_0(N)\longrightarrow E?
\]
Can that additional property be recognized, or even constructed, from peels, their refinements, and their finite covering data?
\end{quote}

This is the geometric form of the problem.  It is close in spirit to the fascination behind dessins d'enfants: an elementary finite drawing may carry arithmetic information which is not elementary at all.  The first part of the paper will therefore stay as long as possible with triangles, coverings, punctures, peels, flips, refinements, and deck transformations.  Only after this geometric picture has been developed will the arithmetic structures be introduced.  The Tate module, Hecke correspondences, Bruhat--Tits trees, Brandt modules, and CM packets are then introduced to ask which additional arithmetic conditions distinguish the modular case.

\section{Introduction: from Belyi to dessins, and from dessins to peels}

The starting point is Belyi's theorem.  A compact Riemann surface is an analytic object: locally it is described by holomorphic coordinates, and globally it belongs to a moduli space of complex structures.  Nevertheless, Belyi discovered that the surfaces which can be defined over $\overline{\Q}$ are exactly those which admit a meromorphic function
\[
        f:X\longrightarrow \PP^1
\]
ramified over at most three points, which may be taken to be $0,1,\infty$ \cite{Belyi}.  Thus an arithmetic condition on an algebraic curve is equivalent to the existence of a branched covering of the sphere with only three prescribed branch values.  I shall call such a pair $(X,f)$ a \emph{Belyi pair}; when both $X$ and $f$ are defined over $\Q$, I shall call it a \emph{rational Belyi pair}.

Grothendieck immediately understood that this theorem was saying much more than its formal statement.  In his \emph{Esquisse d'un programme} he proposed replacing the algebraic curve, for many purposes, by the finite combinatorial object
\[
        \mathcal D=f^{-1}([0,1]),
\]
a bipartite graph embedded in the oriented surface.  He called these objects \emph{dessins d'enfants} \cite{Esquisse}.  The terminology is elementary, but the idea is not.  A dessin remembers the topological covering, its monodromy, the complex structure selected by the covering, and an action of the absolute Galois group
\[
        G_\Q:=\Gal(\overline{\Q}/\Q).
\]
A central feature of Grothendieck's vision is this coexistence of two descriptions: on one side algebraic curves and Galois theory, and on the other side finite graphs drawn on surfaces.

There is a geometric form of the same phenomenon which is especially close to the present paper.  If one decomposes the sphere into two congruent triangles with vertices $0,1,\infty$ and pulls this decomposition back by a Belyi map, the surface acquires a triangulation by congruent triangles.  In the Euclidean equilateral case this point of view was made explicit by Voevodsky and Shabat: a Riemann surface admits the conformal structure coming from an equilateral triangulation exactly in the arithmetic situation predicted by Belyi's theorem \cite{VS}.  Their short 1989 paper was followed by the broader 1990 article \emph{Drawing curves over number fields}, written for the Grothendieck Festschrift \cite{SVDrawing}.  The title expresses perfectly the philosophy: arithmetic curves can literally be drawn.

This circle of ideas was the starting point of my joint work with Jos\'e Juan-Zacar\'ias \cite{JV}; see also his doctoral thesis \cite{JuanThesis}, written under my supervision, especially the chapter on equilateral triangulations, Belyi functions, peels, developing maps, and holonomy.  We studied decorated equilateral triangulations and Belyi functions and introduced a construction which we called a \emph{peel}.  At the end of that work we asked which properties of a dessin characterize elliptic curves defined over $\Q$ and which characterize modular elliptic curves.  The thesis records the possibility of approaching the Taniyama--Shimura theorem through dessins as one of our shared dreams.  The name ``peel'' came from the elementary picture of cutting a triangulated surface until it opens like the peel of an orange.  But there is a better geometric description, and it is the one I shall use here:

There is a second source of motivation which at first seems rather far from Belyi theory.  In my survey on low-dimensional solenoidal manifolds \cite{VerjovskySolenoids}, and more recently in joint work with Fernando Alcalde Cuesta, \`Alvaro Carballido Costas and Matilde Mart\'inez on horocyclic dynamics of hyperbolic solenoidal surfaces of finite type \cite{ACMV}, inverse limits of towers of finite coverings play a basic role.  A solenoid keeps, in one compact laminated space, the information carried by all levels of a covering tower; its Cantor transversal records the corresponding profinite structure.  This suggested to me that one should not ask only what a single peel remembers.  One should also follow it through natural finite covers and then pass to an inverse limit.  Fix a prime $\ell$.  Iterating the multiplication covering $[\ell]:E\to E$ gives
\[
\cdots \xrightarrow{[\ell]} E \xrightarrow{[\ell]} E \xrightarrow{[\ell]} E,
\]
and, after composing the Belyi map with $[\ell^n]$, the compatible Belyi pairs
\[
        (E,f\circ[\ell^n])_{n\geq0}.
\]
I shall call this the \emph{multiplication-covering tower}, or simply the multiplication tower when there is no danger of confusion.  The inverse limit of the covering spaces themselves,
\[
\mathcal S_\ell(E):=\varprojlim\bigl(E\xleftarrow{[\ell]}E\xleftarrow{[\ell]}E\xleftarrow{[\ell]}\cdots\bigr),
\]
is an $\ell$-adic solenoidal lamination over $E$.  The fiber over the origin is canonically
\[
\varprojlim_n E[\ell^n]=T_\ell(E)\cong\Z_\ell^2,
\]
so the Tate module appears geometrically as the profinite transversal of this solenoidal inverse limit.  When $E$ is defined over $\Q$, Galois acts compatibly on these finite fibers and hence on the transversal.  Thus the solenoidal viewpoint does not replace the Tate module; it places it inside the inverse-limit geometry of the covering tower.

\begin{quote}
\emph{A peel is a finite piece of a developing map.  Here the developing map is obtained by continuing the Euclidean triangular charts on the simply connected geometric cover; away from the lifted cone points it is locally an isometry into the Euclidean model, and the ambiguity created by continuation around loops downstairs is measured by the holonomy representation.}
\end{quote}

Indeed, a peel consists of a simply connected finite union $P$ of model triangles and a map
\[
        \Psi:P\longrightarrow X
\]
which is an isometry on every triangle and respects the decoration.  The boundary of $P$ records the identifications needed to reconstruct the surface.  If one kept unfolding triangles indefinitely in the corresponding simply connected geometric cover, one would obtain the usual developing map.  The peel is a finite developed domain adapted to the given finite triangulation.

This formulation records simultaneously the topology, the local Euclidean geometry, and the holonomy of the triangulation.  For a torus, the image of the boundary graph of a peel has the homotopy type of a bouquet of two circles, and therefore represents the two basic homology directions.  More is true.  The peel together with its boundary identifications reconstructs the universal cover of the torus, endowed with the pullback of the singular flat metric.  The deck group---that is, the group of covering transformations---$\pi_1(E)\cong\Z^2$ acts by isometries on the reconstructed universal cover.  Thus a finite peel already contains a periodic infinite geometric object.  In the simplest hexagonal situation the two deck directions are also the period directions of the flat torus.  In general, cone singularities introduce rotational holonomy, and a Euclidean displacement seen in one developed copy need not be a period of the holomorphic one-form of the elliptic curve.  This distinction will be important below.

The first aim of this paper is to put this observation in a precise form.  The singular equilateral metric is naturally described by the cubic differential
\[
        q_f=\frac{(df)^3}{f^2(f-1)^2}.
\]
The peel is a finite developing map for the geometry defined by this differential.  Passing to its canonical cubic-root cover removes the rotational part of the holonomy and produces a translation surface.  In this way the combinatorics of the dessin, the geometry of the peel, and the analytic geometry of differentials fit together without confusing the singular flat metric with the uniformizing complex torus.

The second aim is arithmetic and considerably more ambitious.  The original work with Jos\'e Juan-Zacar\'ias already ended with questions concerning elliptic curves and the Taniyama--Shimura--Weil theorem.  Our question was whether the dessin and the peel contain arithmetic information that distinguishes elliptic curves defined over $\Q$ and, among them, the modular situation.  A representative $\tau$ of the complex isomorphism class alone cannot determine the conductor of an elliptic curve over $\Q$: quadratic twists have the same complex torus and different arithmetic.  By an \emph{arithmetic Belyi pair} I mean a Belyi pair $(X,f)$ defined over a number field (in the applications below, over $\Q$).  Such a pair, together with its Galois action and its system of peels, carries much more structure.

Following that idea leads to two natural operations.  The first is the multiplication-covering tower just described: the maps $[\ell^n]:E\to E$ are finite unramified coverings, and after composition with the Belyi map their relative deck groups are $E[\ell^n]$.  Passing to the inverse limit gives
\[
T_\ell(E):=\varprojlim_n E[\ell^n].
\]
When the Belyi pair is defined over $\Q$, the natural Galois action on every torsion group is compatible with the transition maps and therefore gives the usual action on $T_\ell(E)$.  This has an immediate arithmetic consequence which I want to display rather than leave implicit.  Once the Galois representation on $T_\ell(E)$ is known, its inertia action at a prime $p\neq\ell$ determines the local conductor exponent $f_p(E)$, namely the exponent of $p$ in the conductor $N_E$, while Frobenius at a good prime gives $a_p(E)$.  Here a \emph{good prime} means a prime of good reduction for $E$; a \emph{bad prime} is a prime of bad reduction.  Varying the auxiliary prime $\ell$, the family of multiplication-covering towers, with their compatible Galois actions, therefore determines
\[
N_E=\prod_p p^{f_p(E)},
\]
together with the good Euler factors
\[
1-a_p(E)T+pT^2.
\]
Thus the conductor is determined by the multiplication-covering tower together with its Galois action.  The same tower also records the split/nonsplit local characters at multiplicative primes and, through all the Frobenius traces, the Hasse--Weil $L$-function, that is, the Euler product assembled from these local factors.  Cyclic quotients, on the other hand, produce a system of correspondences on peels with the local geometry of the Bruhat--Tits tree and the algebra of Hecke correspondences.  At a prime $p$ of bad reduction this leads naturally to the finite supersingular Brandt module, namely the degree-zero module generated by isomorphism classes of supersingular elliptic curves in characteristic $p$ (those whose geometric $p$-torsion has no nontrivial points), equipped with the isogeny-counting Brandt operators defined below.  For $q\neq p$, I write $B_q$ for the operator obtained by summing over cyclic isogenies of degree $q$ from each supersingular class.

Thus two apparently different finite pictures emerge from the same elliptic curve.  One contains Frobenius traces,
\[
        a_q(E)=q+1-\#E(\F_q),
\]
and the other carries Brandt operators $B_q$.  The problem is to determine whether these two constructions can be related directly.  The desired identity is
\[
        B_qc_E=a_q(E)c_E
\]
for a nonzero class $c_E$ in the appropriate supersingular module.  In the prime semistable case this reciprocity statement is essentially equivalent to modularity itself.  I shall therefore not disguise the difficulty: this is exactly where the great theorem enters.

The present work does not give a new proof of modularity.  It approaches the modularity question through explicit geometric constructions, records why several direct attempts fail, and studies a construction that does not begin with a modular form.  We shall call the resulting signed supersingular reduction divisors \emph{oriented CM packets}; here CM means complex multiplication.  They are obtained from CM elliptic curves while retaining the local orientation at the conductor prime, and are attached to the Frobenius discriminants
\[
        D_q=a_q(E)^2-4q.
\]
The conductor-$37$ calculation at the end of the paper gives one explicit comparison.  A Frobenius discriminant coming directly from the elliptic curve produces, through oriented CM reduction, the same one-dimensional subspace of the degree-zero supersingular module at $37$ that is also singled out by the local reduction at $37$.  Before reaching that example I isolate a more general point: nonsplit residual Frobenius elements supply a positive-density family of admissible quadratic orders, their discriminants cannot remain bounded, and a Frobenius-conjugate supersingular pair gives a concrete nonvanishing criterion for the anti-invariant packet.

The starting object is still the finite drawing on the surface.  After unfolding the peel one sees the periods and the holonomy.  If the arithmetic structure is retained, the same construction leads to Galois representations, while cyclic isogenies bring in the Hecke operators.  The final question is whether these finite constructions already contain the reciprocity law which modularity asserts.

\paragraph{Organization.}
The geometric part first treats the peel as a finite developing map, the cubic differential of the equilateral structure, and reconstruction of the universal cover and deck lattice.  The arithmetic part then passes from the finite unramified multiplication coverings to their solenoidal inverse limit and its profinite transversal $T_\ell(E)$, and from cyclic quotients to the local lattice correspondences underlying the Hecke and Brandt operators.  The arithmetic part then isolates admissible Frobenius discriminants and proves elementary nonvanishing criteria before returning to the conductor-$37$ example and the reciprocity problem stated near the end.  Throughout, proved statements, computations, interpretations, and conjectures are kept separate.

\section{The peel as a finite developing map}

\subsection{The combinatorial data}
Let
\[
f:X\longrightarrow\PP^1
\]
be a Belyi map.  Choose the standard two-triangle decomposition of the sphere with vertices $0,1,\infty$, and pull it back.  We obtain a decorated triangulation of $X$.  Here the decoration is the data inherited from the standard decomposition, in particular the labels of vertices by $0,1,\infty$ and the distinction between the two triangle types.  If $d=\deg f$, the triangulation is encoded by a transitive monodromy pair
\[
(\sigma_0,\sigma_1)\in S_d\times S_d.
\]
Here $\sigma_0$ and $\sigma_1$ are the permutations of the $d$ sheets obtained by analytic continuation around $0$ and $1$; the monodromy around $\infty$ is $\sigma_\infty=(\sigma_0\sigma_1)^{-1}$.  Transitivity means that the covering $X\to\PP^1$ is connected.  Thus vertices, valences, Euler characteristic, and genus are finite combinatorial data.

Let $G$ be the dual graph of the triangulation.  Choosing a spanning tree $T\subset G$ and unfolding triangles successively across the edges of $T$ produces a simply connected finite union of triangles $P$.  The developing map on this finite union is the peel map.

The precise definition of a peel was given at the beginning of the paper in Definition~\ref{def:peel-overture}.  In the present Euclidean Belyi situation the model triangles are equilateral Euclidean triangles.  A spanning-tree unfolding produces one convenient peel; continuing the same unfolding indefinitely gives the usual developing picture.

\begin{remark}
A spanning-tree unfolding is a convenient way to construct a peel, but I do not regard the spanning tree as the essence of the definition.  The essence is the finite developing map $\Psi:P\to X$ and its boundary identifications.
\end{remark}

\subsection{The boundary graph}
The boundary graph $\mathcal G_P=\Psi(\partial P)$ and its elementary topology were treated in Propositions~\ref{prop:peel-bouquet} and~\ref{prop:peel-H1}.  In genus one it is homotopy equivalent to a bouquet of two circles and carries all of $H_1(X,\Z)$.  We shall use this fact repeatedly without restating it.

\section{The singular equilateral metric and the cubic differential}

This is the first place where one must distinguish topology from complex analysis.

\subsection{A cubic differential on the sphere}
On the sphere with coordinate $z$ consider
\[
q_0=\frac{(dz)^3}{z^2(z-1)^2}.
\]
The metric $|q_0|^{2/3}$ is the flat metric naturally associated with the equilateral two-triangle orbifold, that is, the sphere equipped with the cone data coming from this triangular structure.  Its \emph{linear holonomy}---the rotational part of the Euclidean holonomy obtained by parallel transport around singularities---is contained in the cube roots of unity.

Pulling back by $f$ gives
\[
q_f=f^*q_0=\frac{(df)^3}{f^2(f-1)^2}.
\]

\begin{theorem}[Canonical cubic differential]
Let $x\in f^{-1}\{0,1,\infty\}$, and let $e_x$ be the ramification index of $f$ at $x$.  Then
\[
\ord_x(q_f)=e_x-3.
\]
If $v_x$ is the valence of the corresponding vertex in the triangular decomposition, then $v_x=2e_x$ and
\[
\boxed{\ord_x(q_f)=\frac{v_x-6}{2}.}
\]
\end{theorem}

\begin{proof}
Near a point above $0$, write $f=t^{e_x}$.  Then
\[
q_f\sim\frac{(e_xt^{e_x-1}dt)^3}{t^{2e_x}}
=e_x^3t^{e_x-3}(dt)^3.
\]
The same calculation applies above $1$ and $\infty$.  Around a vertex, black and white triangles alternate, hence $v_x=2e_x$.
\end{proof}

Thus the divisor of the cubic differential is read directly from the dessin.  Writing $\Div(q_f)$ for the formal sum of zeros and poles with their orders,
\[
\Div(q_f)=\sum_x\left(\frac{v_x-6}{2}\right)x.
\]
For a torus this divisor has degree zero, as it must.

\subsection{When the developed peel gives the true period lattice}
If every vertex has valence six, then $q_f$ has neither zero nor pole.  On a torus this implies
\[
q_f=c\eta^3
\]
for a nowhere vanishing holomorphic one-form $\eta$.  In that case the equilateral metric is a smooth flat metric and the developing map is the ordinary developing map of the complex torus.  The translation vectors exposed by the peel are genuine periods.

If some valence differs from six, the situation changes.  The same peel still gives the correct topology, but the developed vectors belong to a singular flat structure with rotational holonomy.  They need not be periods of the holomorphic one-form which uniformizes the elliptic curve.

This is why I prefer to say that a peel is a kind of developing map.  One must always ask: developing map for which geometric structure?

\section{The canonical cubic-root cover}

The rotational holonomy can be removed canonically.

\begin{theorem}[Canonical cover]
Let $(X,q_f)$ be as above.  There is a cyclic cover
\[
\pi:Y\to X
\]
of degree dividing $3$ such that
\[
\pi^*q_f=\omega^3
\]
for a meromorphic one-form $\omega$ on $Y$.  Algebraically one may take the normalization of
\[
u^3=f^2(f-1)^2,
\qquad
\omega=\frac{df}{u}.
\]
In the usual equilateral situation $\omega$ is holomorphic.
\end{theorem}

The cover eliminates the rotational holonomy.  It does \emph{not} in general remove the cone singularities.  A cone point downstairs usually becomes a zero of $\omega$ upstairs.

For example, if a genus-one dessin has four vertices of valence
\[
4,4,8,8,
\]
then
\[
\Div(q_f)=-x_1-x_2+x_3+x_4.
\]
The cubic cover is triply ramified at the four points.  Riemann--Hurwitz gives
\[
2g(Y)-2=4(3-1)=8,
\]
so $g(Y)=5$.  The result is a translation surface of genus five, not a flat torus.  Here a translation surface means a Riemann surface equipped with a holomorphic one-form, whose local integrals give charts with translation transition maps away from the zeros.

The deck group $\mu_3$, the group of cube roots of unity, acts on cohomology and
\[
H^1(Y,\C)^{\mu_3}\cong H^1(X,\C).
\]
Thus the original elliptic curve is still present as the invariant factor.  This gives a correct replacement for the naive idea that one can always read the analytic modulus directly from two Euclidean displacement vectors of the peel.

\section{An example of a peel}
In the degree-three hexagonal example every vertex has valence six, so the developed peel is a genuine flat fundamental domain for the torus.
\begin{figure}[htb]
\centering
\includegraphics[width=.52\textwidth]{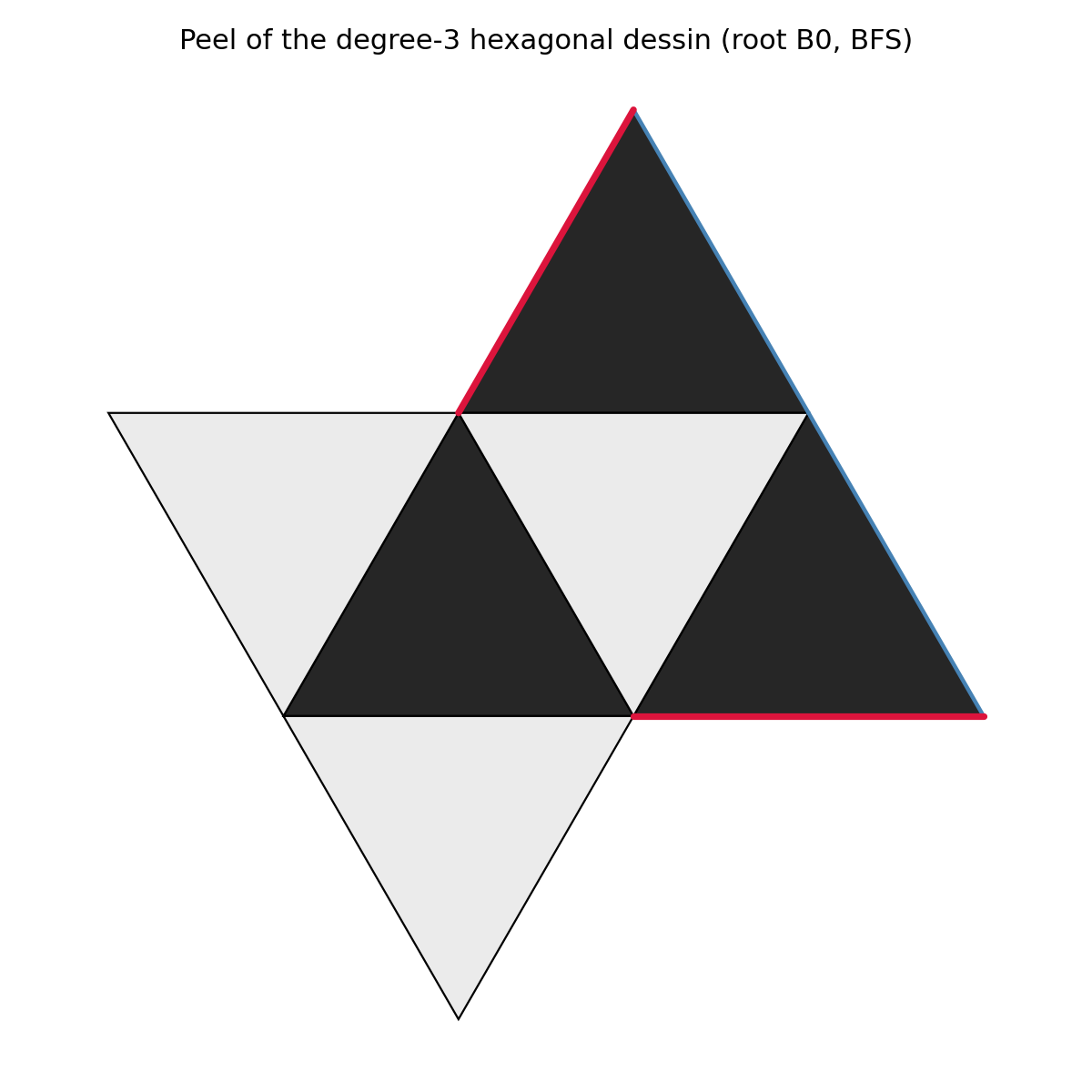}
\caption{A peel in the smooth hexagonal case.}\label{fig:peel-hex}
\end{figure}
If some valence differs from six, the same construction is a finite developing map for a singular flat metric; the canonical cubic-root cover, rather than the displayed side translations alone, is then needed to compare with the analytic periods.

\section{The universal cover reconstructed from a peel and the deck lattice}

The preceding discussion concerns one finite developed domain.  For a torus there is a second construction, already implicit in the original peel picture, which will be useful for the arithmetic part of the paper.

Let
\[
\Psi:P\longrightarrow E
\]
be a peel of a triangulated elliptic curve and put
\[
\mathcal G_P=\Psi(\partial P)\subset E.
\]
By Propositions~\ref{prop:peel-bouquet} and~\ref{prop:peel-H1}, $\mathcal G_P$ has the homotopy type of a bouquet of two circles, the inclusion $\mathcal G_P\hookrightarrow E$ induces an isomorphism on first homology, and
\[
\pi_1(\mathcal G_P)\twoheadrightarrow\pi_1(E)\cong\Z^2.
\]
After collapsing a maximal tree in $\mathcal G_P$, we may choose two loops $a,b$ whose images generate $\pi_1(E)$.  Thus the boundary data of the peel contains generators for the fundamental group; no particular word for the attaching map is needed in what follows.

The same data reconstructs the universal cover.  Take one copy $P_{m,n}$ of $P$ for each $(m,n)\in\Z^2$, and glue boundary pieces according to the side identifications of the original peel.  Crossing a side corresponding to $a$ changes the index by $(1,0)$, while crossing a side corresponding to $b$ changes it by $(0,1)$.  The resulting simply connected surface is naturally identified with the universal cover
\[
\widetilde E\longrightarrow E.
\]
The singular flat metric on $E$ pulls back to a singular flat metric $\widetilde g$ on $\widetilde E$.  The copies of $P$ are glued by the same edge isometries used in the peel, so translation of the index set $\Z^2$ carries each glued triangle isometrically to another one.  Consequently every deck transformation preserves $\widetilde g$.  Hence
\[
\Deck(\widetilde E/E)\cong\pi_1(E)\cong\Z^2
\]
acts by isometries.  The cone points, when they are present, lift to a $\Z^2$-periodic set of cone points.

\begin{proposition}[Universal cover from the peel data]
For a triangulated elliptic curve, the peel together with its boundary identifications determines the universal covering surface equipped with the pullback singular flat metric and its isometric deck action by $\Z^2$.
\end{proposition}

This statement is topological and metric; it should not be confused with analytic uniformization.  If the equilateral metric is smooth, the reconstructed universal cover is the ordinary Euclidean plane and the deck action is by translations.  If cone singularities are present, the universal cover is still topologically $\R^2$, but the pulled-back metric is singular and $\Z^2$ acts by isometries of that singular metric.  The analytic uniformization of the elliptic curve gives a second realization
\[
E\cong\C/\Lambda,
\qquad
\Lambda=\Z\omega_1+\Z\omega_2.
\]
Thus the peel exposes first the abstract deck lattice $\pi_1(E)\cong\Z^2$; uniformization realizes it as the period lattice $\Lambda\subset\C$.  In the notation of \cite{JuanThesis}, the two generators act on the holomorphic universal cover as
\[
z\longmapsto z+\omega_1,
\qquad
z\longmapsto z+\omega_2,
\]
so that $\tau=\omega_1/\omega_2$.

\subsection{Finite quotients of the deck lattice}
Every finite-index subgroup $L\subset\Z^2$ gives a finite unramified covering
\[
\widetilde E/L\longrightarrow E.
\]
The multiplication-by-$\ell^n$ covering corresponds to the subgroup $\ell^n\Z^2$.  Its deck group is therefore
\[
\Z^2/\ell^n\Z^2\cong(\Z/\ell^n\Z)^2.
\]
After choosing the analytic lattice $\Lambda$, the same group is
\[
\Lambda/\ell^n\Lambda
\cong E[\ell^n].
\]
Writing $\Z_\ell$ for the ring of $\ell$-adic integers, the $\ell$-adic completion of the deck lattice is
\[
\varprojlim_n \pi_1(E)/\ell^n\pi_1(E)
\cong
\pi_1(E)\otimes\Z_\ell
\cong\Z_\ell^2.
\]
There are therefore two related inverse limits, and it is useful not to confuse them.  Taking the inverse limit of the \emph{covering spaces} gives the compact solenoidal lamination
\[
\mathcal S_\ell(E)=\varprojlim\bigl(E\xleftarrow{[\ell]}E\xleftarrow{[\ell]}E\xleftarrow{[\ell]}\cdots\bigr).
\]
Projection to the first coordinate makes $\mathcal S_\ell(E)$ a compact two-dimensional lamination over $E$ with totally disconnected fiber.  In the usual terminology for solenoids, such a fiber is a transversal.  The fiber over $0\in E$ is canonically
\[
\varprojlim_n E[\ell^n]=T_\ell(E)\cong\Z_\ell^2.
\]
Thus, in this precise sense, the Tate module is the profinite transversal over the origin of the solenoidal inverse limit.  Taking instead the inverse limit of the \emph{deck groups} gives the same underlying $\Z_\ell$-module.  Arithmetic enters when one retains the compatible $G_\Q$-action on the torsion fibers; this is exactly the usual $\ell$-adic Galois representation.  This is the point at which the solenoidal inverse-limit picture meets the arithmetic information carried by the multiplication-covering tower.

\subsection{Index-$q$ neighbors and isogenies}
The same deck lattice also gives the local finite-index geometry used later to model Hecke correspondences.  Let $q$ be prime.  Cyclic subgroups $C\subset E[q]$ of order $q$ are the lines in
\[
E[q]\cong\F_q^2,
\]
and hence are parametrized by $\PP^1(\F_q)$.  Analytically, if $E=\C/\Lambda$, such a subgroup is equivalent to an intermediate lattice
\[
\Lambda\subset\Lambda_C\subset q^{-1}\Lambda,
\qquad
[\Lambda_C:\Lambda]=q,
\]
with
\[
E/C\cong\C/\Lambda_C.
\]
Therefore the $q+1$ cyclic $q$-isogenies out of $E$ are precisely the $q+1$ index-$q$ neighbors of its rank-two lattice.

After $q$-adic completion this becomes the usual lattice model for the Bruhat--Tits tree of $\PGL_2(\Q_q)$.  Vertices are homothety classes of $\Z_q$-lattices in $\Q_q^2$ (two lattices are homothetic if one is obtained from the other by multiplication by a nonzero element of $\Q_q$), and adjacent vertices may be represented by lattices $L,L'$ with
\[
qL\subset L'\subset L,
\qquad [L:L']=q.
\]
Thus the deck lattice reconstructed from the peel already contains the local finite-index pattern that underlies the degree-$q$ Hecke correspondence.  This statement is local: the passage from that lattice pattern to an actual Hecke operator still requires the global correspondence module introduced below.

The three passages
\[
\pi_1(E)\cong\Z^2,
\qquad
\pi_1(E)\otimes\Z_\ell,
\qquad
\pi_1(E)\otimes\Z_q
\]
should therefore be kept together.  The first is the deck lattice reconstructed from the peel.  The second is its $\ell$-adic completion and is the underlying module of the Tate module; the arithmetic information is the compatible Galois action.  The third is the $q$-adic lattice geometry whose index-$q$ neighbors model the local Hecke correspondence.  The automorphic step is the passage from this local lattice geometry to a global space, such as modular forms or the supersingular Brandt module, on which the Hecke correspondences act.

\section{Arithmetic information from the multiplication-covering tower}

The peel reconstructs the abstract deck lattice and all its finite-index quotients, but this is still geometric information.  It does not determine the conductor.  Quadratic twists are isomorphic over $\C$ and may have different conductors over $\Q$.  Arithmetic enters through the rational model and, more specifically, through the Galois action on the finite quotients of the deck lattice.

As above, a Belyi pair is a smooth projective curve together with a Belyi map; here it is defined over $\Q$.  Write
\[
G_\Q=\operatorname{Gal}(\overline{\Q}/\Q)
\]
for the absolute Galois group of $\Q$.  Let $(E,\beta)/\Q$ be a genus-one Belyi pair, and fix a prime $\ell$.  For every $n\geq0$ define
\[
\beta_n=\beta\circ[\ell^n].
\]
Since multiplication by $\ell^n$ is finite \emph{\'etale} (equivalently here, finite and unramified) in characteristic zero, $\beta_n$ is again a Belyi map.

The relative deck transformations of
\[
(E,\beta_n)\longrightarrow(E,\beta)
\]
are translations by points of $E[\ell^n]$.  Therefore
\[
\Deck(\beta_n/\beta)\cong E[\ell^n].
\]
The $\ell$-adic Tate module of $E$ is, by definition,
\[
T_\ell(E)=\varprojlim_n E[\ell^n],
\]
where the transition maps are multiplication by $\ell$.  Passing to the inverse limit of deck groups gives the following exact statement.

\begin{proposition}[Multiplication-covering tower and Tate module]
For a rational Belyi pair $(E,\beta)$ and a prime $\ell$,
\[
\boxed{
\varprojlim_n\Deck(\beta\circ[\ell^n]/\beta)
\cong T_\ell(E).
}
\]
This is an isomorphism of $G_\Q$-modules.
\end{proposition}

\begin{proof}
The transition map from level $n+1$ to level $n$ is induced by multiplication by $\ell$ on torsion.  Hence the inverse system is exactly
\[
\cdots\to E[\ell^{n+1}]\xrightarrow{[\ell]}E[\ell^n]\to\cdots\to E[\ell].
\]
For $P\in E[\ell^n]$ write $t_P$ for translation by $P$.  For $\sigma\in G_\Q$,
\[
\sigma t_P\sigma^{-1}=t_{\sigma P}.
\]
Thus the induced Galois action is the usual one on the Tate module.
\end{proof}

Consequently the multiplication-covering tower, together with its compatible Galois action, determines the representation
\[
\rho_{E,\ell}:G_\Q\to\GL_2(\Z_\ell).
\]
At a good prime $q\neq\ell$, let $\Frob_q$ denote an arithmetic Frobenius element modulo inertia, characterized on the residue field by $x\mapsto x^q$.  The characteristic polynomial of its action is
\[
X^2-a_q(E)X+q,
\]
so in particular
\[
a_q(E)=\Tr\rho_{E,\ell}(\Frob_q),
\qquad
\det\rho_{E,\ell}(\Frob_q)=q.
\]
This conclusion uses only the Galois action on the Tate module; no modularity theorem is involved.

\begin{theorem}[Arithmetic data determined by the multiplication-covering tower]\label{thm:peel-remembers}
Let $(E,\beta)/\Q$ be an elliptic Belyi pair, and consider the family of multiplication-covering towers
\[
\bigl(E,\beta\circ[\ell^n]\bigr)_{n\geq0}
\]
as $\ell$ varies over the primes.  With their natural Galois action these towers determine:
\begin{enumerate}[(i)]
\item the representations $\rho_{E,\ell}$ on $T_\ell(E)$;
\item for every prime $p$, the local conductor exponent $f_p(E)$ and hence the conductor
\[
N_E=\prod_p p^{f_p(E)};
\]
\item for every good prime $p$, the Frobenius trace $a_p(E)$ and the local Euler polynomial
\[
1-a_p(E)T+pT^2;
\]
\item the split or nonsplit character at every prime of multiplicative reduction;
\item consequently, the Hasse--Weil $L$-function of $E$, namely the Euler product determined by these local polynomials together with the bad-prime local factors; with the local Galois data included, the complete collection of local factors is determined.
\end{enumerate}
In particular, the conductor is determined by the multiplication-covering tower together with its Galois action, independently of modularity.
\end{theorem}

\begin{proof}
The proposition identifies the inverse limit with $T_\ell(E)$ as a $G_\Q$-module.  For $p\neq\ell$, the restriction to a decomposition group at $p$ determines inertia, the Artin conductor exponent, and Frobenius on inertia invariants; these give the conductor, the good-prime Euler factors, and the split/nonsplit character at multiplicative primes.  Varying $p$ and an auxiliary $\ell\neq p$ gives the stated data.
\end{proof}

\begin{remark}
The conductor appears here through the Galois action on the multiplication-covering tower.  Later, in the already modular example $X_0(11)\to X(1)$, the same kind of integer can be read directly from side-pairing matrices of a Schreier peel; that finite observation should not be confused with the general recovery above.
\end{remark}

\section{Hecke correspondences from elliptic peels}

The second construction is obtained by quotienting by cyclic subgroups.

Fix a prime $q$.  The $q+1$ lines
\[
L\subset E[q]
\]
are the cyclic subgroups of order $q$.  For every such $L$ there is an isogeny
\[
\pi_L:E\to E/L.
\]
Starting from a Belyi pair $(E,\beta)$ one may pull the triangulation back by $[q]:E\to E$ and then, for each cyclic subgroup $L\subset E[q]$, pass to the quotient $E/L$.  For each $L$, the map $\beta\circ[q]$ is invariant under translation by $L$ and therefore descends to a Belyi map on $E/L$; in this way one obtains $q+1$ quotient Belyi pairs and their corresponding peels.

The formal sum of these quotients is the classical cyclic-isogeny correspondence of degree $q$.

Here $\Z_q$ and $\Q_q$ denote the $q$-adic integers and numbers.  The local combinatorics is the Bruhat--Tits tree of $\PGL_2(\Q_q)$: its vertices are homothety classes of rank-two $\Z_q$-lattices.  Neighbors correspond to index-$q$ inclusions, equivalently to lines in a two-dimensional $\F_q$-space.  Thus each vertex has $q+1$ neighbors.

This is the point where the peel construction meets Hecke geometry for a structural reason rather than by analogy:
\[
\fbox{
\parbox{0.78\textwidth}{
\centering
the peel exposes a rank-two lattice,\\
and Hecke acts through finite-index lattice correspondences.
}}
\]

One must, however, use the correct correspondence module.  Here a push--pull operator means the usual operator obtained by pulling a class back along one map of a finite correspondence and then pushing it forward along the other, with the natural degree multiplicities.  I write $T_q$ for the push--pull operator attached to the degree-$q$ cyclic-isogeny correspondence and, more generally, $T_m$ for the corresponding Hecke operator of index $m$.  If one simply counts paths in the unweighted graph, the usual Hecke normalization is lost.  With these push--pull weights one recovers the usual Hecke algebra.  In the weight-two, trivial-character normalization used later,
\[
T_mT_n=T_{mn}\quad((m,n)=1),
\]
and for $q$ away from the level,
\[
T_{q^{r+1}}=T_qT_{q^r}-q\,T_{q^{r-1}}.
\]
The correction term is the contribution of the $q$ backtracking choices in the local lattice tree.  This is why ordinary unweighted path counting is not the Hecke operator.

\section{The two peel operations and the Eichler--Shimura comparison}

Two operations will be used below.  Pullback by the finite unramified multiplication coverings
\[
(E,\beta)\longmapsto(E,\beta\circ[\ell^n])
\]
keeps the elliptic curve fixed and records its torsion together with the Galois action, while cyclic quotients $E\to E/L$ move through the isogeny class and realize the index-$q$ lattice neighbors underlying the local Hecke correspondence.  Iterating the latter yields the Bruhat--Tits tree of $\PGL_2(\Q_q)$; the multiplication map $[q]$ corresponds to scaling the lattice by $q$ and therefore gives the radial scalar step in that local lattice picture.  Weighted push--pull, rather than unweighted path counting, gives the usual Hecke normalization.

At a good prime $q$, the Galois action on the multiplication-covering tower gives
\[
X^2-a_q(E)X+q,
\]
whereas Eichler--Shimura on the modular side gives
\[
F_q^2-T_qF_q+q=0.
\]
Thus the missing comparison is concrete: one must construct a one-dimensional Brandt--Hecke eigenspace $L$ for which
\[
T_q|_L=a_q(E)
\]
for every good $q$.  Producing that line independently is essentially the modularity problem addressed later.

\section{A modular Schreier peel}

When modularity is already present, a peel can retain visible congruence information.  For
\[
X_0(11)\longrightarrow X(1),
\]
the dessin is the Schreier graph for the action of $\PSL_2(\Z)$ on $\PP^1(\F_{11})$.  Unfolding a coset graph gives a peel whose side pairings lie in $\Gamma_0(11)$.  For a convenient transversal one obtains, for example,
\[
\begin{pmatrix}1&-1\\0&1\end{pmatrix},\qquad
\begin{pmatrix}4&-3\\11&-8\end{pmatrix},\qquad
\begin{pmatrix}3&-2\\11&-7\end{pmatrix}.
\]
The common divisibility by $11$ in the lower-left entries exhibits the level.  This is not a method for recovering the conductor of an arbitrary elliptic curve: the modular group is already built into the presentation.  It only shows that passing to a peel need not erase congruence data.

\section{The supersingular Brandt module at a prime conductor}

From this section through the prime-conductor Brandt discussion, $p$ is a prime conductor and, when an elliptic curve $E/\Q$ is being compared with the Brandt module, $E$ is assumed semistable of conductor $p$ unless another hypothesis is stated explicitly.  Thus $E$ has multiplicative reduction at $p$ and good reduction at every prime $q\neq p$ used for the Hecke and Brandt operators.

Let $p$ be prime and let $\mathcal S_p$ be the set of supersingular elliptic curves over $\overline{\F}_p$, considered up to isomorphism.  Recall that an elliptic curve in characteristic $p$ is supersingular if its geometric $p$-torsion has no nontrivial points (equivalently, its endomorphism algebra is the quaternion algebra ramified at $p$ and $\infty$).  Define the degree-zero supersingular module
\[
X_p^0
=
\left\{
\sum_{s\in\mathcal S_p}n_s[s]:\sum_s n_s=0
\right\}.
\]
For $q\neq p$ define the Brandt operator
\[
B_q[s]
=
\sum_{C\subset E_s,\ |C|=q}[E_s/C].
\]
This is exactly the degree-$q$ isogeny correspondence on the finite supersingular set.

Classically, $X_p^0$ is identified with the character lattice of the toric part of the N\'eron model of $J_0(p)$ at $p$, where $J_0(p)$ denotes the Jacobian of the modular curve $X_0(p)$.  Here the N\'eron model is the smooth group scheme over $\Z_p$ extending the Jacobian with its universal mapping property, and the character lattice of its toric part is $X^*(T)=\operatorname{Hom}(T,\mathbf G_m)$.  In the comparisons with an elliptic curve $E$ made below, we use this toric character group only when $E$ has multiplicative reduction at $p$; in particular this holds for a semistable elliptic curve of prime conductor $p$.  No uniform statement about additive reduction or its component group is being used.  Up to the standard automorphism weights---reciprocal factors accounting for the sizes of the automorphism groups of the supersingular elliptic curves, with the exceptional extra automorphisms at $j=0$ and $j=1728$---the Brandt operators above give the Hecke action on this lattice.

Thus there are now two independently defined arithmetic objects:
\[
\begin{array}{ccl}
\text{multiplication-covering tower of }E\text{ with its Galois action}
&\rightsquigarrow&
T_\ell(E),\ a_q(E),\\[1mm]
\text{supersingular Brandt module at }p
&\rightsquigarrow&
X_p^0,\ B_q.
\end{array}
\]
The modularity problem is the problem of making them meet.

\section{Frobenius, Atkin--Lehner, and the eigenspaces}

Arithmetic Frobenius acts on $\mathcal S_p$.  Every supersingular $j$-invariant lies in $\F_{p^2}$, so Frobenius acts as an involution on the set and therefore on $X_p^0$.
Write
\[
X_p^0\otimes\Q=X_p^{0,+}\oplus X_p^{0,-}.
\]

For prime level $p$, the Fricke involution (equivalently the Atkin--Lehner involution $w_p$) on $X_0(p)$ is induced on the upper half-plane by
\[
w_p:\tau\longmapsto -\frac{1}{p\tau},
\qquad
\begin{pmatrix}0&-1\\ p&0\end{pmatrix},
\]
whose projective action normalizes $\Gamma_0(p)$.  On the dual graph of the special fiber of $X_0(p)$---the graph having one vertex for each irreducible component of the fiber and one edge for each intersection point of two components---this involution and Frobenius differ by a sign, with the precise convention depending on whether one works with edges, characters, or homology.  In the convention used here the relation is
\[
w_p=-F
\]
on the character module.

For an elliptic curve with multiplicative reduction at $p$,
\[
a_p(E)=
\begin{cases}
+1,&\text{split multiplicative},\\
-1,&\text{nonsplit multiplicative}.
\end{cases}
\]
The corresponding modular newform has Atkin--Lehner sign $w_p=-a_p(E)$.  Therefore the expected one-dimensional Brandt eigenspace has Frobenius eigenvalue $a_p(E)$.

This already has content when one of the eigenspaces is one-dimensional.  In that case every Brandt operator preserves a one-dimensional space and hence acts there by a scalar.  The split/nonsplit multiplicative-reduction sign at the bad prime---equivalently $a_p(E)=+1$ or $-1$---then selects a full simultaneous Brandt eigenline, that is, a one-dimensional subspace on which every $B_q$ acts by a scalar, before one computes any good-prime eigenvalue.

\section{The example of conductor $37$}

This example is where most of the later questions first appear explicitly.

\subsection{The supersingular set}
In characteristic $37$ the supersingular $j$-invariants are
\[
s_0=8,
\qquad
s_\pm=3\pm\sqrt{15}.
\]
Frobenius fixes $s_0$ and exchanges $s_+$ and $s_-$.  Therefore
\[
X_{37}^{0,+}=\Q\bigl(-2[s_0]+[s_+]+[s_-]\bigr),
\]
and
\[
X_{37}^{0,-}=\Q\bigl([s_+]-[s_-]\bigr).
\]
Both Frobenius eigenspaces are one-dimensional.

Thus split and nonsplit multiplicative reduction select, respectively, primitive integral generators of the two one-dimensional eigenspaces
\[
c_+=(-2,1,1),
\qquad
c_-=(0,1,-1).
\]
This conclusion uses only Frobenius on the supersingular set and the local reduction sign.

\subsection{The $2$-Brandt matrix}
A direct calculation of the $2$-isogeny adjacency gives, in the ordered basis
\[
[s_0],[s_+],[s_-],
\]
\[
B_2=
\begin{pmatrix}
1&1&1\\
1&0&2\\
1&2&0
\end{pmatrix}.
\]
The constant vector has eigenvalue $3$.  On degree zero,
\[
B_2c_-=-2c_-,
\qquad
B_2c_+=0.
\]

For the curve
\[
E_a:y^2+y=x^3-x
\]
direct point counting gives
\[
a_2(E_a)=-2.
\]
For the other conductor-$37$ isogeny class one obtains
\[
a_2=0.
\]
Thus the first Brandt calculation matches the two arithmetic Frobenius traces.

This is an experiment, not a proof of the modularity theorem.  The important observation is that the local decomposition into the $+1$ and $-1$ Frobenius eigenspaces already explains why the two relevant simultaneous eigenlines have the shapes above.

\section{The reciprocity conjecture}

We can now state the central problem in its shortest form.  Recall that an elliptic curve over $\Q$ is \emph{semistable} if at every prime it has either good or multiplicative reduction.  By \emph{modularity} of $E$ I mean the existence of a weight-two normalized newform (a normalized Hecke eigenform new at that level) of level equal to the conductor of $E$ whose Hecke eigenvalues agree with the Frobenius traces $a_q(E)$ at the good primes.

\begin{conjecture}[Reciprocity conjecture]
Let $E/\Q$ be semistable of prime conductor $p$.  There exists a nonzero class
\[
c_E\in X_p^0\otimes\Q
\]
which lies in the corresponding Frobenius eigenspace (the $+1$ or $-1$ Frobenius eigenspace determined by the split/nonsplit reduction of $E$ at $p$) and satisfies
\[
\boxed{B_qc_E=a_q(E)c_E}
\]
for every prime $q\neq p$.
\end{conjecture}

The equality is simple to write.  It is not simple to prove.  In fact it is essentially modularity itself.

\[
\fbox{
\parbox{0.78\textwidth}{
\centering
\textbf{The key step toward the reciprocity conjecture.}
Construct directly from the Belyi peel and its arithmetic data a Hecke action
whose eigenvalues are the Frobenius traces $a_\ell(E)$,
without assuming modularity.
}}
\]

\begin{theorem}[Why the reciprocity conjecture is equivalent to modularity]
For a semistable elliptic curve $E/\Q$ of prime conductor $p$, the reciprocity conjecture is equivalent, up to the standard Jacquet--Langlands and Faltings identifications, to modularity of $E$.
\end{theorem}

\begin{proof}
Suppose $c_E$ exists.  The definite quaternion algebra ramified at $p$ and $\infty$ realizes the Brandt module as a quaternionic automorphic module.  The Jacquet--Langlands correspondence transfers the simultaneous Brandt eigensystem of $c_E$---that is, the collection of eigenvalues of the commuting Brandt operators on the line $\Q c_E$---to a weight-two eigenform $f$ of level $p$.  By construction
\[
a_q(f)=a_q(E)
\]
for every good prime $q$.  The corresponding two-dimensional $\ell$-adic representations therefore have the same Frobenius characteristic polynomials on a density-one set.  The Chebotarev density theorem then identifies the semisimplified Galois representations, and Faltings's isogeny theorem gives an isogeny between $E$ and the elliptic modular quotient attached to $f$.

Conversely, if $E$ is modular, a quotient
\[
J_0(p)\twoheadrightarrow E
\]
induces, contravariantly on toric character groups at $p$, an injection
\[
X^*(T_E)\hookrightarrow X^*(T_{J_0(p)})\cong X_p^0.
\]
Since $X^*(T_E)\cong\Z$ for multiplicative reduction, the image of $1$ is a nonzero vector $c_E$.  Hecke equivariance gives
\[
B_qc_E=a_q(E)c_E.
\]
\end{proof}

This theorem is useful because it prevents circular arguments.  If one constructs the map
\[
X^*(T_E)\to X_p^0
\]
by using the modular quotient, one has not produced a new proof.  An independent argument must construct $c_E$ without starting from $J_0(p)\to E$.

\section{Three obstructions}

Three elementary obstructions delimit the problem.  First, a correspondence obtained only by pulling back a correspondence from $\PP^1\times\PP^1$ through Belyi maps acts trivially on the $H^1\otimes H^1$ part, since $H^1(\PP^1)=0$; the Belyi sphere alone therefore cannot create the missing elliptic correspondence.  Second, the monodromy permutations of two unrelated Belyi maps do not canonically identify their rank-two homology lattices.  Third, at conductor $37$ an unoriented Gross vector is Frobenius invariant and therefore misses the anti-invariant eigenspace.  The last point is what leads to the oriented CM construction below.

\section{Frobenius discriminants and CM packets}

This suggests the following experiment.

For a good prime $q$ define the Frobenius discriminant
\[
D_q=a_q(E)^2-4q<0.
\]
Let $\pi_q$ denote the Frobenius endomorphism of the reduction $E_{/\F_q}$.  Its characteristic polynomial is $X^2-a_q(E)X+q$, and the order $\Z[\pi_q]$ has discriminant $D_q$.  When the reduction is ordinary (that is, not supersingular), $\Z[\pi_q]$ is an order in the imaginary quadratic field $\Q(\pi_q)$; its index in the full endomorphism ring accounts for the usual square factor relating the two order discriminants.

Now fix the bad prime $p$.  An elliptic curve has \emph{complex multiplication} (CM) by an order $\mathcal O_D$ if its endomorphism ring over $\C$ contains (and in the present discussion equals) that imaginary quadratic order.  CM elliptic curves with endomorphism by $\mathcal O_D$ have reductions at primes above $p$, and when $p$ is inert or ramified in the relevant quadratic field these reductions may be supersingular.  Equivalently, one may describe the same data through optimal embeddings of $\mathcal O_D$ into maximal orders of the quaternion algebra ramified at $p$ and $\infty$.

\begin{definition}[CM packet]
For a negative discriminant $D$ and a prime $p$, the \emph{CM packet of discriminant $D$ at $p$} is the finite weighted divisor on the supersingular set in characteristic $p$ obtained by reducing CM elliptic curves with endomorphism ring $\mathcal O_D$ at primes above $p$.  Equivalently, it may be described by counting optimal embeddings of $\mathcal O_D$ into the maximal orders corresponding to the supersingular classes, with the usual automorphism weights, meaning reciprocal factors which compensate for the sizes of the automorphism groups of the corresponding elliptic curves (with the exceptional extra automorphisms at $j=0$ and $j=1728$).
\end{definition}

If one forgets all local orientation data, this is the CM divisor usually called a Gross packet; see, for example,~\cite{Gross}.  There is also an intrinsic theory of oriented supersingular elliptic curves, in which an orientation is an embedding of an imaginary quadratic field into the rational endomorphism algebra; see Onuki~\cite{Onuki}.  For the concrete Brandt calculations in this paper I keep the following lighter convention.  For every Frobenius orbit $\{s,F(s)\}$ of size two in the supersingular set, choose one member and call it $s^+$; call the other $s^-=F(s^+)$.  An \emph{oriented reduction class} is a reduction class together with this choice of sign.  Thus Frobenius-conjugate reductions are counted separately rather than identified.  Changing all choices exchanges the two signs and changes an anti-invariant class by an overall sign, but does not change the rational line that it spans.  The intrinsic orientation will be used below only to prove that Frobenius cannot fix an oriented CM reduction when $p$ is inert.  The remaining passage from intrinsic orientations to the ordinary Brandt module is kept explicit rather than hidden in the notation.

\begin{definition}[Oriented CM packet]
For a negative discriminant $D$ and a prime $p$, let
\[
R_{p,D}^{\mathrm{or}}
\]
denote the oriented refinement of the CM packet of discriminant $D$ at $p$.  If $F$ denotes the Frobenius involution on the supersingular module, define
\[
R_{p,D}^{\pm}=(1\pm F)R_{p,D}^{\mathrm{or}},
\]
and, when a degree-zero class is needed, replace a divisor $D=\sum_s d_s[s]$ by
\[
D^0=D-\frac{\deg D}{h_p}\sum_{s\in\mathcal S_p}[s],
\qquad
\deg D=\sum_s d_s,\quad h_p=\#\mathcal S_p.
\]
Thus $D^0\in X_p^0\otimes\Q$.
\end{definition}

The precise integral normalization depends on these automorphism weights.  For the present conceptual discussion only the resulting rational line matters.

\section{Admissible Frobenius discriminants and nonvanishing}\label{sec:admissible}

The preceding construction leaves two elementary questions.  First, which Frobenius discriminants can actually contribute CM reductions to the quaternion algebra ramified at $p$ and $\infty$?  Second, when does the resulting anti-invariant class fail to vanish after the orientation is forgotten?  Both questions have useful answers before any modularity input is used.

Let $E/\Q$ be semistable of prime conductor $p>2$.  The action of $G_\Q=\Gal(\Qbar/\Q)$ on the $p$-torsion $E[p]$ gives the residual representation
\[
        \overline\rho_{E,p}:G_\Q\longrightarrow\GL_2(\F_p).
\]
Let $q\neq p$ be a good prime and put
\[
        D_q=a_q(E)^2-4q.
\]
For an integer $a$ prime to $p$, $\left(\frac{a}{p}\right)$ denotes the Legendre symbol: it is $+1$ when $a$ is a square modulo $p$ and $-1$ when it is a nonsquare.
I shall call $q$ \emph{$p$-admissible} when $p\nmid D_q$ and
\[
        \left(\frac{D_q}{p}\right)=-1.
\]
Thus $p$ is inert in the quadratic field $\Q(\sqrt{D_q})$.

\begin{lemma}[Residual criterion for admissibility]\label{lem:residual-admissible}
For a good prime $q\neq p$, the characteristic polynomial of $\overline\rho_{E,p}(\Frob_q)$ is
\[
        X^2-a_q(E)X+q\pmod p,
\]
and its discriminant is $D_q\pmod p$.  If $p\nmid D_q$, then $q$ is $p$-admissible if and only if this polynomial is irreducible over $\F_p$.
\end{lemma}

\begin{proof}
The trace and determinant of Frobenius on $E[p]$ are $a_q(E)$ and $q$ modulo $p$, so the displayed polynomial is immediate.  A quadratic polynomial over the odd finite field $\F_p$ is irreducible exactly when its discriminant is a nonsquare.
\end{proof}

This gives a simple Galois-theoretic source of admissible CM orders.

\begin{proposition}[Positive density of admissible Frobenius primes]\label{prop:positive-density}
Assume that the image of
\[
        \overline\rho_{E,p}:G_\Q\longrightarrow\GL_2(\F_p)
\]
contains an element whose characteristic polynomial is irreducible over $\F_p$.  Then the set of $p$-admissible good primes has positive natural density.  Here a set $\mathcal P$ of primes has natural density $\delta$ if
\[
\lim_{x\to\infty}\frac{\#\{q\in\mathcal P:q\le x\}}{\#\{q\text{ prime}:q\le x\}}=\delta.
\]
\end{proposition}

\begin{proof}
Let $L=\Q(E[p])$.  In $\Gal(L/\Q)$ consider the union $C$ of conjugacy classes whose images under $\overline\rho_{E,p}$ have irreducible characteristic polynomial.  By hypothesis $C$ is nonempty.  The Chebotarev density theorem gives density $|C|/|\Gal(L/\Q)|>0$ for the unramified primes with Frobenius in $C$.  Removing the finitely many primes of bad reduction and the prime $p$ does not change the density.  Lemma~\ref{lem:residual-admissible} identifies these primes with the desired admissible set.
\end{proof}

The discriminants obtained in this way cannot stay in a finite list.

\begin{proposition}[Unbounded admissible discriminants]\label{prop:unbounded-D}
Let $\mathcal Q$ be any positive-density set of good primes.  Then the set
\[
        \{D_q=a_q(E)^2-4q:q\in\mathcal Q\}
\]
is unbounded in absolute value.  In particular, under the hypothesis of Proposition~\ref{prop:positive-density}, there is a sequence of $p$-admissible primes $q_n$ with $D_{q_n}\to-\infty$.
\end{proposition}

\begin{proof}
Fix $M>0$.  If $|D_q|\le M$ and $q\le x$, then for one of the finitely many integers $D$ with $-M\le D<0$ we have
\[
        q=\frac{a_q(E)^2-D}{4}.
\]
Hasse's bound gives $|a_q(E)|\le2\sqrt q\le2\sqrt x$.  For a fixed $D$, each integer value of $a_q(E)$ determines at most one value of $q$, so there are $O(\sqrt x)$ such primes.  Summing over the finitely many $D$ gives
\[
 \#\{q\le x:q\in\mathcal Q,\ |D_q|\le M\}=O_M(\sqrt x).
\]
A positive-density set of primes has order $x/\log x$, and $\sqrt x=o(x/\log x)$.  Hence the discriminants cannot remain bounded.
\end{proof}

There are two different meanings of nonvanishing, and it is useful not to mix them.  The first lives in the intrinsically oriented supersingular set.  If $K$ is imaginary quadratic, an orientation on a supersingular elliptic curve $S$ is an embedding
\[
        \iota:K\hookrightarrow\End^0(S).
\]
An isogeny transports the orientation by conjugation.  This is the standard oriented language used, for example, in~\cite{Onuki}.

\begin{theorem}[Oriented Frobenius nonvanishing]\label{thm:oriented-nonvanishing}
Let $K$ be an imaginary quadratic field in which $p$ is inert, and let $(S,\iota)$ be a supersingular elliptic curve in characteristic $p$ equipped with a $K$-orientation.  Let $F(S,\iota)$ denote its Frobenius transform with the transported orientation.  Then
\[
        (S,\iota)\not\simeq F(S,\iota)
\]
as oriented elliptic curves.  Consequently
\[
        [(S,\iota)]-[F(S,\iota)]\neq0
\]
in the free module on oriented supersingular classes.
\end{theorem}

\begin{proof}
Suppose that an orientation-preserving isomorphism
\[
        u:F(S,\iota)\longrightarrow(S,\iota)
\]
exists.  Compose $u$ with the $p$-power Frobenius isogeny $\pi:S\to S^{(p)}$.  Because the orientation on $S^{(p)}$ is transported by $\pi$ and $u$ preserves it, the endomorphism
\[
        \alpha=u\circ\pi\in\End(S)
\]
commutes with $\iota(K)$.  The centralizer of the quadratic field $\iota(K)$ in the quaternion algebra $\End^0(S)$ is $\iota(K)$ itself.  Hence $\alpha\in\iota(K)$.  On the other hand
\[
        \deg\alpha=p.
\]
For an element of an imaginary quadratic field the degree is its field norm.  Thus $K$ would contain an element of norm $p$.  This is impossible when $p$ is inert, since the $p$-adic valuation of the norm of every element of $K^\times$ is even.  The contradiction proves the assertion.
\end{proof}

Theorem~\ref{thm:oriented-nonvanishing} does not by itself give a nonzero vector in the ordinary Brandt module, because forgetting the orientation can identify a curve with its Frobenius transform.  The exact extra condition is elementary.

\begin{proposition}[Nonvanishing after forgetting the orientation]\label{prop:ordinary-nonvanishing}
Let $D<0$ be a quadratic-order discriminant for which CM reduction modulo $p$ is supersingular.  Suppose that one CM class of discriminant $D$ reduces to a supersingular class $s$ with
\[
        F(s)\neq s.
\]
Then
\[
        [s]-[F(s)]\neq0\qquad\hbox{in }X_p^0,
\]
and it belongs to $X_p^{0,-}$.  In particular, an oriented packet in which this Frobenius orbit occurs with nonzero weight has nonzero anti-invariant part.
\end{proposition}

\begin{proof}
The two basis elements $[s]$ and $[F(s)]$ are distinct, so their difference is nonzero and has degree zero.  Since every supersingular $j$-invariant is defined over $\F_{p^2}$,
\[
        F([s]-[F(s)])=[F(s)]-[F^2(s)]=-[s]+[F(s)],
\]
which is the required $-1$ eigenrelation.
\end{proof}

The proposition has a convenient finite test.  Let $H_D(X)\in\Z[X]$ be the Hilbert class polynomial of the order $\mathcal O_D$, namely the monic polynomial whose complex roots are the $j$-invariants of elliptic curves with endomorphism ring $\mathcal O_D$.

\begin{corollary}[Hilbert class polynomial criterion]\label{cor:hilbert-nonvanishing}
If $H_D(X)\bmod p$ has an irreducible quadratic factor whose roots are supersingular, then the corresponding oriented CM reduction has nonzero projection to $X_p^{0,-}$.
\end{corollary}

\begin{proof}
The two roots are $s$ and $s^p$ in $\F_{p^2}\setminus\F_p$, so Proposition~\ref{prop:ordinary-nonvanishing} applies.
\end{proof}

The preceding statements are elementary except for Chebotarev.  There is also a useful connection with the much deeper equidistribution theory of CM points.  Herrero--Menares--Rivera-Letelier prove that, along a fixed $p$-adic discriminant class and under their stated hypotheses, CM measures with discriminant tending to $-\infty$ converge to a probability measure supported on the corresponding formal-CM locus~\cite{HMRL}.  Here the formal-CM locus means the subset of the $p$-adic moduli space singled out by the prescribed endomorphisms of the associated formal group; I use the term only in this standard sense from their paper.  Aka--Luethi--Michel--Wieser prove, under additional congruence hypotheses, surjectivity and equidistribution of simultaneous supersingular reductions for sufficiently large CM discriminant~\cite{ALMW}.

I record only the consequence that is needed here, with the support condition stated explicitly.

\begin{proposition}[Equidistribution route to nonvanishing]\label{prop:equid-route}
Let $D_n\to-\infty$ be discriminants in a fixed $p$-adic discriminant class satisfying the hypotheses of the equidistribution theorem of Herrero--Menares--Rivera-Letelier, and let $\nu$ be the limiting measure.  Suppose the support of $\nu$ meets the residue disk of a supersingular class $s$ with $F(s)\neq s$.  By the residue disk of $s$ I mean the inverse image of $s$ under reduction from the $p$-adic moduli space to characteristic $p$.  Then, for all sufficiently large $n$ in a subsequence, a CM point of discriminant $D_n$ reduces to $s$; in particular the corresponding oriented reduction has nonzero anti-invariant projection to $X_p^0$.
\end{proposition}

\begin{proof}
Choose an open neighborhood $U$ of a formal CM point in the residue disk of $s$ small enough to remain in that disk.  Since this point lies in the support of $\nu$, one has $\nu(U)>0$.  For the open set $U$, weak convergence (equivalently, the open-set inequality in the Portmanteau theorem) gives
\[
        \liminf_n\mu_{D_n}(U)\ge\nu(U)>0,
\]
where $\mu_{D_n}$ is the normalized CM measure.  Hence $U$ contains a CM point of discriminant $D_n$ for all sufficiently large $n$ along a subsequence.  Its reduction is $s$, and Proposition~\ref{prop:ordinary-nonvanishing} finishes the argument.
\end{proof}

\begin{remark}\label{rem:no-overclaim-equid}
Proposition~\ref{prop:equid-route} deliberately does not assert that the required support condition holds for every sequence of Frobenius discriminants $D_q$.  Proving that assertion in the generality needed here would require additional input.  The point is narrower: Propositions~\ref{prop:positive-density} and~\ref{prop:unbounded-D} supply infinitely many large admissible discriminants from the multiplication-covering tower together with its Galois action, while CM equidistribution gives a precise mechanism by which such discriminants can force nonvanishing once their limiting formal-CM support meets a non-rational supersingular residue class.  This separates a genuine nonvanishing problem from the much stronger Hecke-eigenvalue reciprocity problem.
\end{remark}

\section{The oriented CM calculation at conductor $37$}

Consider again
\[
E_a:y^2+y=x^3-x.
\]
Direct point counts give
\[
a_{19}(E_a)=0.
\]
Therefore
\[
D_{19}=a_{19}(E_a)^2-4\cdot19=-76.
\]
Moreover $37$ is inert in $\Q(\sqrt{-19})$, equivalently
\[
        \left(\frac{-76}{37}\right)=-1.
\]
Thus $q=19$ is $37$-admissible in the sense of Section~\ref{sec:admissible}, and it is the prime used here to construct the CM packet.  It should be distinguished from the prime $2$, at which the Brandt eigenvalue is checked below.  Indeed $D_2=a_2(E_a)^2-8=-4$ and $37$ splits in $\Q(i)$, so $q=2$ is not admissible for constructing this CM packet.

The class number---the number of proper ideal classes---of the order of discriminant $-76$ is $3$.  The corresponding Hilbert class polynomial, whose roots are the $j$-invariants of the CM elliptic curves with endomorphism ring of discriminant $-76$, therefore has degree three.  In the computation which motivated this section, its reduction modulo $37$ factors with roots exactly the three supersingular invariants
\[
s_0,\ s_+,\ s_-.
\]
Thus the three CM classes reduce one-for-one to the three supersingular classes.

Frobenius fixes $s_0$ and exchanges $s_+$ and $s_-$.  Hence the anti-invariant oriented part is nonzero and necessarily lies on
\[
X_{37}^{0,-}=\Q([s_+]-[s_-]).
\]
Consequently
\[
\boxed{R_{37,-76}^-\text{ is a nonzero rational multiple of }[s_+]-[s_-].}
\]

This is the most concrete construction in the present investigation that does not start from modularity:
\[
\begin{aligned}
E_a &\xrightarrow{\ a_{19}=0\ } D=-76
      \longrightarrow \mathcal O_{-76},\\
\mathcal O_{-76} &\longrightarrow
\text{oriented CM reduction mod }37
\longrightarrow [s_+]-[s_-].
\end{aligned}
\]

Nothing in this construction starts from a modular parametrization of $E_a$.

Because the anti-invariant Brandt eigenspace at $37$ is one-dimensional, every $B_q$ acts on this line by a scalar.  In particular the direct $B_2$ calculation gives
\[
B_2R_{37,-76}^-=-2R_{37,-76}^-.
\]
Since $a_2(E_a)=-2$, one obtains the concrete identity
\[
B_2R_{37,-76}^-=a_2(E_a)R_{37,-76}^-.
\]
Again, this is an experiment at one level and for one Hecke operator, namely $B_2$.  Its interest is that the relevant one-dimensional Brandt eigenspace was obtained from a Frobenius discriminant of $E$, not identified after first knowing the modular form.

\begin{remark}
The Frobenius discriminants $D_q=a_q(E)^2-4q$ do not satisfy the standard Gross-vector Hecke recurrence $D\mapsto Dq^2$; consequently no such recurrence is used below.  Any compatible relation must retain the oriented CM data.
\end{remark}

\section{A more precise reciprocity problem}

The $37$ experiment suggests a more precise question than the reciprocity conjecture.

\begin{conjecture}[Oriented CM-packet reciprocity]
Let $E/\Q$ be semistable of prime conductor $p$, and let
\[
D_q=a_q(E)^2-4q
\]
for good primes $q$.  Form the oriented CM reduction packets
\[
R_{p,D_q}^{\varepsilon_p}\in X_p^0\otimes\Q,
\]
where $\varepsilon_p=a_p(E)=\pm1$ selects the Frobenius sign.  Then the Hecke module generated by these packets contains a nonzero simultaneous Brandt eigenline $L_E$ satisfying
\[
B_\ell|_{L_E}=a_\ell(E)
\]
for all $\ell\neq p$.
\end{conjecture}

There is a stronger form, closer to what one would need for a proof.

\begin{problem}[Peel reciprocity problem]
Construct, using only the arithmetic Belyi peel, its multiplication-covering towers, Frobenius data, and oriented CM reduction, a nonzero class
\[
c_E\in X_p^0
\]
for which
\[
B_qc_E=a_q(E)c_E
\]
for every good prime $q$, without invoking modularity in the construction.
\end{problem}

A solution would give a new proof of prime-conductor semistable modularity.  The theorem of the preceding section explains why: once the eigenline is constructed, Jacquet--Langlands and Faltings finish the comparison.

\section{One-dimensional eigenspaces and the Fricke quotient}

At conductor $37$ the Frobenius eigenspaces in $X_{37}^0\otimes\Q$ are one-dimensional.  More generally, under the standard identification of the supersingular character module with the toric character group of $J_0(p)$, the relevant Frobenius eigenspace has the same dimension as the corresponding eigenspace of the Fricke involution
\[
w_p:\tau\longmapsto-\frac1{p\tau}
\]
on weight-two differentials.  In particular, when the associated Fricke quotient $X_0^+(p)=X_0(p)/\langle w_p\rangle$ has genus one, the split/nonsplit multiplicative-reduction sign at $p$ selects a unique one-dimensional Hecke line.  The conductor-$37$ example is of this type; at higher dimension the good-prime traces are needed to separate eigensystems.

\section{Main results and remaining questions}

We now collect the principal results established above and indicate the remaining arithmetic questions suggested by them.  The paper establishes the cubic differential attached to the equilateral structure, the canonical cubic-root cover, reconstruction of the universal cover and deck lattice from a peel, the $\ell$-adic solenoidal inverse limit of the multiplication-covering tower and the identification of its profinite transversal with $T_\ell(E)$, the local lattice interpretation of cyclic-isogeny Hecke correspondences, and the identification, at prime conductor, of the degree-zero supersingular module with the toric character lattice of $J_0(p)$ together with its Brandt--Hecke action.  It also gives a residual criterion for admissible Frobenius discriminants, proves that a residual image containing an element with irreducible characteristic polynomial over $\F_p$ supplies a positive-density set of admissible primes, and proves that their discriminants are unbounded.  On the CM side, Theorem~\ref{thm:oriented-nonvanishing} shows that inert Frobenius cannot fix an intrinsically oriented supersingular CM class, while Proposition~\ref{prop:ordinary-nonvanishing} and Corollary~\ref{cor:hilbert-nonvanishing} give a finite criterion for nonvanishing after forgetting the orientation.  Proposition~\ref{prop:equid-route} records the precise additional support condition under which $p$-adic CM equidistribution forces such nonvanishing.  Finally, the paper proves that an independently constructed class
\[
c_E\in X_p^0\otimes\Q,\qquad B_qc_E=a_q(E)c_E,
\]
would imply modularity, while modularity supplies such a class.

At conductor $37$ the supersingular set, Frobenius action, and the $2$-Brandt matrix are computed explicitly, and for the nonsplit isogeny class the discriminant $D_{19}=-76$ gives an oriented CM packet whose anti-invariant part lies on the expected Brandt eigenspace.

The principal remaining arithmetic question can now be stated precisely:
\begin{problem}[Peel reciprocity]\label{prob:peel-reciprocity-final}
Construct $c_E$ from the Belyi peel, its multiplication-covering towers, Frobenius data, and oriented CM reduction, without using a modular quotient, and prove
\[
B_qc_E=a_q(E)c_E
\]
for every good prime $q$.
\end{problem}
A closely related problem is to compare the intrinsic orientations of Theorem~\ref{thm:oriented-nonvanishing} functorially with the weighted Brandt module, including the automorphism weights at $j=0,1728$, and then to determine the transformation of the resulting packets under the Brandt operators.  A more modest intermediate problem is to remove the support hypothesis from Proposition~\ref{prop:equid-route} for a sequence of Frobenius discriminants $D_q$.

\section{Concluding remarks}

The paper starts with a finite triangulated picture.  In genus one a peel and its boundary identifications recover the deck lattice of the elliptic curve.  For each prime $\ell$, the finite unramified coverings $[\ell^n]:E\to E$ give a compatible tower.  The inverse limit of the covering spaces is an $\ell$-adic solenoidal lamination over $E$, and its profinite transversal over the origin is $T_\ell(E)$; when the rational structure is retained, the compatible Galois action on this transversal gives $\rho_{E,\ell}$.  Cyclic quotients $E\to E/L$ supply the index-$q$ lattice neighbors that underlie the local Hecke correspondence, while at the conductor prime the supersingular Brandt module provides the finite module in which the desired reciprocity is expressed by the eigenvalue relation above.

The conductor-$37$ calculation shows concretely how a Frobenius trace produces an admissible quadratic order whose oriented CM reduction lands in the expected one-dimensional Brandt eigenspace.  The general results of Section~\ref{sec:admissible} explain why two different primes naturally occur in this example: the good prime $19$ supplies the Frobenius discriminant used to form the CM packet, while the different good prime $2$ is used to check one Brandt eigenvalue.  The positive-density result for admissible Frobenius primes, together with the unboundedness of their discriminants, gives an infinite supply of candidate CM packets.  The oriented and ordinary nonvanishing criteria then isolate exactly what must happen before one even asks for the Hecke eigenvalue identity.  This still does not prove the reciprocity law.  The remaining equality $B_qc_E=a_q(E)c_E$ for all good $q$ is the modularity barrier.

\section*{Funding}
This work was supported by Proyecto PAPIIT, Direcci\'on General de Asuntos del Personal Acad\'emico, Universidad Nacional Aut\'onoma de M\'exico [grant number IN103324].

\section*{Acknowledgments}
The notion of a peel and the original questions concerning peels of equilateral triangulations were developed in joint work with Jos\'e Juan-Zacar\'ias.  I am grateful for that collaboration.  The present paper continues the subject in a different arithmetic direction.

I also gratefully acknowledge the use of the large language models ChatGPT (OpenAI) and Claude (Anthropic) for proofreading, checking calculations, and helping to organize the exposition.  The mathematical statements, conjectures, and responsibility for their correctness are the author's.

\end{document}